\documentclass[11pt]{amsart}

\usepackage{amsmath}
\usepackage{amssymb}
\usepackage{graphicx}

\newtheorem{thm}{Theorem}[section]

\newtheorem{lem}[thm]{Lemma}
\newtheorem{prop}[thm]{Proposition}
\newtheorem{definition}[thm]{Definition}
\newtheorem{rem}[thm]{Remark}

\newcommand{\bd}{\begin{displaymath}}
\newcommand{\ed}{\end{displaymath}}
\newcommand{\be}{\begin{equation}}
\newcommand{\ee}{\end{equation}}
\newcommand{\bea}{\begin{eqnarray}}
\newcommand{\eea}{\end{eqnarray}}
\newcommand{\bda}{\begin{eqnarray*}}
	\newcommand{\eda}{\end{eqnarray*}}
\newcommand{\ba}{\begin{array}}
	\newcommand{\ea}{\end{array}}

\newcommand{\mt}{\mapsto}
\newcommand{\st}{\subset}

\newcommand{\sm}{\setminus}

\newcommand{\B}{{\bf B}}

\newcommand{\U}{{\mathcal U}}
\newcommand{\V}{{\mathcal V}}
\newcommand{\W}{{\mathcal W}}

\newcommand{\HH}{{\mathcal H}}

\newcommand{\al}{\alpha}

\newcommand{\sth}{ \, :\;}

\newcommand{\dd}{\mbox{\rm\,d}}
\newcommand{\gr}{\mbox{\rm gph\,}}

\newcommand{\Epi}{\mbox{\rm Epi}\,}

\newcommand{\lmi}{\mathop{\mbox{\rm liminf}}\limits}

\renewcommand{\lll}{\langle}
\newcommand{\rrr}{\rangle}

\newcommand{\p}{{\mathcal P}}

\def\R{\mathbb{R}}
\newcommand{\To}{\rightrightarrows }

\newcommand{\dex}{{\delta x}}

\newcommand{\deu}{{\delta u}}
\newcommand{\dep}{{\delta p}}

\newcommand{\bino}{\bigskip\noindent}

\usepackage[margin=3.2cm]{geometry}

\numberwithin{equation}{section}

\title[Strong bi-metric regularity in affine optimal control problems]{Strong bi-metric regularity in affine optimal control problems}

\author[A. Dom\'inguez Corella]{Alberto Dom\'inguez Corella}\thanks{This research is supported by 
	the Austrian Science Foundation (FWF) under grant No P31400-N32,
	by  the Air Force Office of Scientific Research under award number FA9550-18-1-0254 , 
	by the ``FMJH Program Gaspard Monge in optimization and operation research'' 
	(PGMO 2015-2832H, PGMO 2016-1570H and PGMO 2018-0047H), and from the support to this 
	program from EDF and EDF-THALES-ORANGE-CRITEO}
\address[A. Dom\'inguez Corella]{Institute of Statistics and Mathematical Methods in Economics,
	Vienna University of Technology, Austria}
\email{{\tt alberto.corella@tuwien.ac.at}}

\author[M. Quincampoix]{Marc Quincampoix}
\address[M. Quincampoix]{Univ Brest,  UMR CNRS 6205, Laboratoire de Mathématiques de Bretagne Atlantique 	
	6, Avenue Victor Le Gorgeu, 	29200 Brest, 	France}
\email{\tt Marc.Quincampoix@univ-brest.fr}

\author[V.M. Veliov]{Vladimir M. Veliov}
\address[V.M. Veliov]{Institute of Statistics and Mathematical Methods in Economics,
	Vienna University of Technology, Austria}
\email{\tt vladimir.veliov@tuwien.ac.at}



\begin{document}

\begin{abstract}
The paper presents new sufficient conditions for the property of strong bi-metric regularity of the optimality map
associated with an optimal control problem which is affine with respect to the control variable ({\em affine problem}).
The optimality map represents the system of first order optimality conditions (Pontryagin maximum principle), 
and its regularity is of key importance for the qualitative and numerical analysis of optimal control problems.
The case of affine problems is especially challenging due to the typical discontinuity of the optimal control functions.
A remarkable feature of the obtained sufficient conditions is that they do not require convexity of the objective functional. 
As an application, the result is used for proving uniform convergence of the Euler discretization method 
for a family of affine optimal control problems.
\end{abstract}

\maketitle


\section{Introduction} \label{SIntro}

Regularity properties of the system of first order necessary optimality conditions 
for optimization problems play a key role in qualitative 
analysis and reliable numerical treatment of such problems 
(see e.g. the books \cite{Bonnans+Shapiro,Klatte+Kummer,AD+TR-book2,Ioffe}). 
For optimal control problems, the investigation of regularity properties
of the map associated with the Pontryagin maximum principle (called further {\em optimality map}) 
was first initiated in \cite{DH-93}, which deals with problems that satisfy the so-called coercivity condition. 
The latter, however is never fulfilled for problems in which both dynamics and cost are affine with respect 
to the control (further called {\em affine problems}).   
Results about strong metric sub-regularity of the optimality map for affine problems were obtained in 
the recent papers \cite{PSV,Osmol+Vel-19}. The property of {\em strong metric regularity} 
(see e.g. \cite[Chapter 3]{AD+TR-book2}) of the optimality map
proved to be important for convergence and error estimates of numerical methods (discretizations, 
gradient projection, Newton method, etc.). However, more suitable for affine problems
is a specific extension the strong metric regularity introduced in \cite{QV2013} under the name 
{\em Strong bi-Metric Regularity} (Sbi-MR). The present paper
investigates this property for Lagrange-type affine optimal control problems of the form
\be \label{Eg}
\min \Big\{ J(u) := \int_0^T  [w(t,x(t)) + \lll s(t,x(t)),  u(t)  \rrr ] \dd t \Big\},
\ee
subject to
\be \label{Ex}
\dot x(t) = a(t,x(t))+B(t,x(t)) u(t) ,  \quad  x(0)=x^0,
\ee
\be  \label{Eu}
u(t) \in U, \quad t \in [0,T].
\ee
Here the state vector $x(t)$ belongs to $\R^n$,  the control function $u$ has values $u(t)$ that belong
to a given set $U$ in $\R^m$ for almost every (a.e.)  $t \in [0, T]$.
Correspondingly, $w$ is a scalar function on $[0,T] \times \R^n$, $s$ is an $m$-dimensional vector function 
($\lll \cdot , \cdot \rrr$ denotes the scalar product), $a$ and $B$ are vector-/matrix-valued functions with appropriate 
dimensions. 
The initial state $x^0$ and the final time $T>0$ are fixed.
The set of feasible control functions $u$, denoted in the sequel by $\U$, consists of all Lebesgue measurable
and bounded functions $u: [0,T] \to U$. Accordingly,  the state trajectories $x$, that are solutions of \eqref{Ex}
for feasible controls, are Lipschitz continuous functions on $[0,T]$.

It is well known that the Pontryagin (local) maximum principle can be written in the form of a
generalized equation
\be \label{EFF}
0 \in F(y),
\ee
where $y = (x(\cdot),u(\cdot),p(\cdot))$ encapsulates the state function $x(\cdot)$,
the control function $u(\cdot) \in \U$, and
the adjoint (co-state) function $p(\cdot)$, and the inclusion $0 \in F(y)$ represents the state equation,
the co-state equation, and the maximization condition in the maximum principle
(the last being the inclusion of the derivative of the associated Hamiltonian with respect
to the control in the normal cone to $\U$ at $u(\cdot)$).
The detailed definition of the mapping $F$ in (\ref{EFF}), called further {\em optimality map} 
is given in the next section.

In the next paragraphs we remind the definition of Sbi-MR in the form used in \cite{PSV} and \cite{QSV-18}. 
Let $(Y,d_Y)$,  $(Z,d_Z)$, $(\tilde{Z},d_{\tilde Z})$ be metric spaces,
with $\tilde{Z}\subset Z$ and $d_Z \leq d_{\tilde Z}$  on $\tilde{Z}$.\footnote%
{\label{Fn_d<d} This inequality can be understood as $d_Z(z) \leq c \, d_{\tilde Z}(z)$ for every $z \in \tilde Z$,
	where $c$ is a constant.}
Denote by $\B_Y(\hat y;a)$, $\B_Z(\hat z;b)$ and $\B_{\tilde Z}(\hat z;b)$ the closed balls in the metric spaces
$(Y,d_Y)$, $(Z, d_Z)$ and $(\tilde{Z},d_{\tilde Z})$ with radius $a>0$ or $b>0$ centered at $\hat{y}$ and $\hat{z}$,
respectively.

Given a set-valued map
$\Phi:Y \rightrightarrows Z$, $\gr \Phi:=\lbrace (y,z)\in Y\times Z:\; z\in\Phi(y)\rbrace $
denotes the graph of $\Phi$. The inverse map, $\Phi^{-1}:Z\rightrightarrows Y$,
is the set-valued map defined as $\Phi^{-1}(z):=\lbrace y\in Y: z\in \Phi(y)\rbrace$.

\begin{definition}\label{DSbiMR}
	The set-valued map $\Phi:Y \rightrightarrows Z$ is {\em strongly bi-metrically regular} (Sbi-MR)
	(with disturbance space $\tilde Z$)
	at $\hat{y}\in Y$ for $\hat{z} \in \tilde{Z}$
	with constants $\kappa \geq 0$, $a>0$ and $b>0$, if $(\hat{y},\hat{z}) \in {\rm graph} (\Phi)$
	and the following properties are fulfilled:
	
	(i) the map
	$\B_{\tilde Z}(\hat z;b) \ni z \mapsto \Phi^{-1}(z) \cap \B_Y(\hat{y};a)$ is single-valued;
	
	(ii) for all $z, z'\in \B_{\tilde Z}(\hat{z};b)$
	\be\label{Ereg}
	d_Y(\Phi^{-1}(z)\cap \B_Y(\hat{y};a),\Phi^{-1}(z')\cap \B_Y(\hat{y};a)) \leq \kappa d_Z(z,z').
	\ee
\end{definition}

\bino
We stress that the difference between this notion and the standard notion of strong metric regularity (see e.g. 
\cite[Chapter 3]{AD+TR-book2}) is that the ``disturbances'' $z$ have to belong to the smaller space, $\tilde Z$
(with the bigger distance), but the Lipschitz property in (ii) holds with respect to the smaller distance, $d_Z$,
in the right-side of (\ref{Ereg}). A detailed explanation of the reasons for the appropriateness of this definition
is given in \cite[Introduction]{QSV-18}.

Sufficient conditions for more specific problems and some applications of the Sbi-MR property 
are presented in \cite{PSV} and \cite{QSV-18}. 
The main aim of the present paper is to obtain new, more general, sufficient conditions for 
Strong bi-Metric Regularity (Sbi-MR)
of the optimality map $F$ in an appropriate space setting. A new feature of these conditions is that they involve not
only the second derivative of the associated Hamiltonian with respect to the control, but also its first derivative.
Thanks to that, they may be also fulfilled for problems 
with a non-convex objective functional, which is a new founding in the optimal control context, in general. 

We present the sufficient  conditions for Sbi-MR in Section~\ref{S_suff_cond}
and give a detailed proof in Section~\ref{SProof}. In Section \ref{Sspec} we specialize these conditions 
to the case of affine problems with bang-bang solutions and give an example where they apply to a non-convex problem. 
As an application, in Section \ref{Sunif_discr} we prove that the obtained sufficient conditions imply uniform 
first order convergence of the Euler discretization scheme when applied to affine problems that are close enough
to a reference one. This result is of importance, for example, for the justification of  Model Predictive Control methods 
applied to affine problems.


\section{Sufficient conditions for strong bi-metric regularity} \label{S_suff_cond}

We will use the following standard notations. 
The euclidean norm and the scalar product in $\R^n$ (the elements of which are regarded 
as column-vectors) are denoted by $|\cdot|$ and $\lll \cdot,\cdot \rrr$, respectively. The transpose
of a matrix (or vector) $E$ is denoted by $E^\top$. For a function $\psi :\R^p \to \R^r$ of the variable 
$z$ we denote by $\psi_z(z)$ its derivative (Jacobian), represented by an $(r \times p)$-matrix.
If $r=1$, $\nabla_z \psi(z) = \psi_z(z)^\top$ denotes its gradient (a vector-column of dimension $p$). 
Also for $r=1$, $\psi_{zz}(z)$ denotes the second derivative (Hessian), represented by a $(p \times p)$-matrix.
For a function $\psi :\R^{p + q} \to \R$ of the variables $(z,v)$, $\psi_{zv}(z,v)$ denotes its mixed second
derivative, represented by a $(p \times q)$-matrix. The space $L^k([0,T],\R^r)$, with $k = 1, 2$ 
or $k = \infty$, consists of all (classes of equivalent) Lebesgue measurable $r$-dimensional
vector-functions defined on the interval $[0,T]$, for which the standard norm $\|\cdot\|_k$ is finite. 
Often the specification $([0,T],\R^r)$ will be omitted in the notations. 
As usual, $W^{1,k} = W^{1,k}([0,T],\R^r)$ denotes the space of absolutely continuous
functions $x:[0,T] \to \R^r$ for which the first derivative belongs to $L^k$.
The norm in $W^{1,k}$ is defined as $\| x \|_{1,k} := \| x \|_k  + \| \dot x \|_k$.
Moreover, $\B_X(x;r)$ will denote the ball of radius $r$ centered at $x$ in a metric space $X$. 

Allover the paper we use the abbreviation 
\be \label{Efg}
f(t,x,u) = a(t,x)+B(t,x) u, \qquad g(t,x,u) = w(t,x) + \lll s(t,x),  u \rrr.
\ee
For problem (\ref{Eg})--(\ref{Eu}) we make the following assumption.

\bino
{\em Assumption (A1).} 
The set $U$ is convex and compact; the functions $f: \R \times \R^n \times \R^m \to \R^n$
and $g: \R \times \R^n \times \R^m \to \R$ have the form as in (\ref{Efg}) and
are two times differentiable in $(t,x)$, and the derivatives
are Lipschitz continuous\footnote{\label{FNA1}The assumption of {\em global} Lipschitz 
	continuity is made for convenience. Since the analysis in this paper is local (in a neighborhood of a reference 
	trajectory $\hat x(\cdot)$ in the uniform metric), it can be replaced with {\em local} Lipschitz continuity.}. 

\bino
Define the Hamiltonian associated with problem (\ref{Eg})--(\ref{Eu}) as usual:
\bd
H(t,x,p,u) := g(t,x,u) + \lll p, f(t,x,u) \rrr,\quad  p\in\R^n.
\ed

The local form of the Pontryagin maximum (here minimum) principle for problem (\ref{Eg})-(\ref{Eu})
can be represented by the following optimality system for $(x,u)$ and an absolutely continuous (here Lipschitz) function
$p : [0,T] \to \R^n$: for a.e. $t \in [0,T]$
\bea
\label{EPx}          0 &=&   - \dot x(t) + f(t,x(t),u(t)), \quad x(0) - x^0 = 0, \\
\label{EPp}        0 &=&  \dot p (t) + \nabla_{\!\! x} H(t,x(t),p(t),u(t)), \quad p(T) = 0 ,\\
\label{EPu}       0 &\in& \nabla_u H(t,x(t),p(t),u(t)) + N_U(u(t)), 
\eea
where the normal cone $N_U(u)$ to the set $U$ at $u \in \R^m$ is defined in the usual way,
\bd
N_U(u) = \left\{ \begin{array}{ll} \{y \in \R^n  \mid {\langle y, v-u \rangle} \leq 0
	{\text{ for all } v \in U}\} & {\text{if } u \in U,}\\
	\emptyset & \text{{otherwise.}}\end{array}\right.
\ed

Assumption (A1) implies that there exists a number $M > 0$ such that for any $u \in \U$ the corresponding solution $x$
of (\ref{EPx}) and also the solution  $p$ of (\ref{EPp}) exist on $[0,T]$ and
\be \label{EM}
\max \{ |x(t)|, \,   |\dot x(t)|, \,  |p(t)|, \, |\dot p(t)| \} \leq M \quad \mbox{for a.e. $t \in [0,T]$}.
\ee
In what follows, $\bar M$ will be any number larger that $M$.

Let us introduce the metric spaces
\bd
Y := \{(x,p,u) \in W^{1,1}\times W^{1,1}\times  L^1 : \, 
x(0) = x^0, \, p(T) = 0\}\cap\hspace*{0.04cm}\B^2_{W^{1,\infty}}(0; \bar M)\times\mathcal U.
\ed
and
\bd
Z := L^\infty \times L^\infty \times L^\infty  \mbox{ and } \quad
\tilde Z :=  L^\infty \times L^\infty \times W^{1,\infty} \st Z.
\ed
The distances in these spaces are induced by norms, therefore we keep the norm-notations:
for $y = (x,p,u) \in Y$
\bd
\| y \| := \| x \|_{1,1} +  \| p \|_{1,1} + \| u \|_{1} 
\ed
and for $z = (\xi,\pi,\rho)$ in $Z$ or in $\tilde Z$, respectively,
\bd
\| z \|_Z := \|\xi\|_1 + \|\pi\|_1+\|\rho\|_\infty, \qquad
\| z \|_{\sim} := \|\xi\|_{\infty}+\|\pi\|_{\infty} +\|\rho\|_{1,\infty}.
\ed
Notice that $Y$ is a complete metric space, thanks to the compactness of the set $U$.

Now, we define the set-valued mapping $F : Y \To Z$ as 
\be \label{EOM}
F(y) := \left( \begin{array}{c}
	- \dot x + f(\cdot,x,u) \\
	\dot p + \nabla_{\!\! x} H(\cdot,y) \\
	\nabla_{\!u} H(\cdot,y)
\end{array} \right) 
+ \left( \begin{array}{c}
	0 \\
	0 \\
	N_\U(u)
\end{array} \right),  
\ee
where $N_\U(u)$ is the normal cone to the set $\U$ of admissible controls at $u$, considered as a subset of $L^\infty$:
\bd
N_{\U}(u) := \begin{cases} \emptyset & \mbox{if } u\notin \U \\
	\{ v \in L^\infty \sth v(t)  \in N_U(u(t))  \mbox{ for a.e. $t \in [0,T]$} \} & \mbox{if } u\in \U.
\end{cases}
\ed  
Notice that $F(Y) \st Z$, and $\nabla_{\!u} H(\cdot,y) \in \tilde Z$
thanks to the affine structure of the problem, namely, the independence of $\nabla_{\!u} H(\cdot,y)$ of $u$.

\bino
With these definitions, the necessary optimality conditions (\ref{EPx})--(\ref{EPu}) take the form 
\be \label{EOS}
F(y) \ni 0,
\ee      
therefore $F$ is called {\em optimality map} associated with problem (\ref{Eg})--(\ref{Eu}).
The main result in this paper is a sufficient condition for Sbi-MR of the optimality mapping $F: Y \To Z$ 
with perturbation space $\tilde Z$. To do this we fix a reference solution $\hat y = (\hat x, \hat p, \hat u)$.
We mention that such always exists since on assumption (A1) problem  (\ref{Eg})--(\ref{Eu}) has a solution.
To shorten the notations  we skip arguments with ``hat" in
functions, shifting the ``hat" on the top of the notation of the
function, so that $\hat f(t) := f(t,\hat x(t),\hat u(t))$,
$\hat s(t) = s(t,\hat x(t))$, $\hat H(t) := H(t,\hat x(t),\hat u(t),\hat p(t))$,
$\hat H(t,u) := H(t,\hat x(t),u,\hat p(t))$, etc. Moreover, denote

\begin{align*}
	\hat A(t) := f_x(t,&\hat x(t),\hat u(t)), \quad
	\hat B(t) := f_u(t,\hat x(t),\hat u(t)) =B(t,\hat x(t))\\
	&\hat \sigma(t) := \nabla_u \hat H(t) = \hat B(t)^\top  \hat p(t) + \hat s(t).
\end{align*}

Let us introduce the following functional of $L^1 \ni \deu \mt \Gamma(\deu) \in \R$:
\be \label{EGamma}
\Gamma(\deu) :=\int_0^T \left[ \lll\hat H_{xx}(t) \delta x(t), \dex(t) \rrr +
2 \lll \hat H_{ux}(t) \dex(t), \deu(t) \rrr \right] \dd t,  
\ee
where $\dex$ is the solution of the equation $\dot{\dex} = \hat A \dex + \hat B \deu$ 
with initial condition $\dex(0) = 0$.


\bino
{\em Assumption (A2).}
There exist numbers $c_0$, $\al_0 > 0$ and $\gamma_0 > 0$ such that
\bd
\int_0^T \lll \sigma(t), \deu(t) \rrr  \dd t +  \Gamma(\deu) \geq c_0 \| \deu \|_1^2,
\ed 
for every $\deu = u' - u$ with $u', u \in \U  \cap B_{L^1}(\hat u; \al_0)$, and
for every function $\sigma \in \B_{W^{1,\infty}}(\hat \sigma; \gamma_0)\cap (-N_\U(u))$.

\bino
Assumption (A2) will be analyzed and discussed in details in Section \ref{Sspec}. 
Now we formulate the main theorem.

\begin{thm} \label{TSbi-MR}
	Let Assumption (A1) be fulfilled for problem (\ref{Eg})--(\ref{Eu}) 
	and let $\hat y = (\hat x, \hat p, \hat u)$ be a solution 
	of the optimality system  (\ref{EOS}) (with $F$ defined in (\ref{EOM}))
	for which Assumption (A2) is fulfilled. Let, in addition, the matrix $\hat H_{ux}(t) \hat B(t)$ 
	be symmetric for a.e. $t \in[0,T]$.
	Then the optimality map $F : Y \To Z$ is strongly bi-metrically regular at 
	$\hat y$ for zero with disturbance space $\tilde Z \st Z$.    
\end{thm}


\section{Proof of the main result} \label{SProof}

The proof of Theorem \ref{TSbi-MR} consists of several steps.

{\bf Step 1.} The following result (adapted to the present problem formulation, assumptions, and notations) 
was proved in \cite[Theorem 3.1]{QSV-18}.\footnote{%
	\label{FnQSV} A Mayer problem is considered in \cite{QSV-18}, but the result also applies to Lagrange problems after 
	a standard transformation. Moreover, the assumptions in \cite{QSV-18} are somewhat weaker than 
	(A1).}

\begin{thm} \label{T_QSV}
	Let the assumptions in Theorem  \ref{TSbi-MR} be satisfied.
	Then strong bi-metric regularity of the set-valued map $y \mt F(y)$
	at $\hat y$ for $0$ (in the spaces as in Theorem  \ref{TSbi-MR}) is equivalent to the strong
	bi-metric regularity of the map $y \mt L(y)$, at $\hat y$ for $0$, where 
	\bda
	L(y) = \left( \begin{array}{c}
		- \dot x + \hat f +\hat A (x - \hat x)  + \hat B (u - \hat u) \\
		\dot p + \nabla_{\!\! x} \hat H + \hat H_{xy} (y - \hat y) \\
		\nabla_{\!u} \hat H  + \hat H_{uy}(y - \hat y) +  N_\U(u)
	\end{array} \right).
	\eda
\end{thm}

The map $L$ represents the partial linearization of $F$ around $\hat y = (\hat x, \hat p, \hat u)$.
Thanks to the identity $\hat H_{uu} = 0$, $L$ maps $Y$ to $Z$, 
and moreover, $\hat y$ solves the inclusion $L(\hat y) \ni 0$. 

To shorten the notations, we set for this section (skipping the dependence on $t$)
\bd
W :=  \hat H_{xx}, \;\; S :=  \hat H_{ux}, \;\; A := \hat A =  \hat f_x, \;\; B = \hat B = B(\hat x).
\ed
We remind the already introduced notation $\hat \sigma =  \nabla_{\!u} \hat H$. 
Then, also having in mind the identity $\hat H_{uu} = 0$, we can recast the definition of $L(y)$ as
\bd
L(y) = \left( \begin{array}{c}
	- \dot x + \dot{\hat x} + A (x - \hat x)  + B (u - \hat u) \\
	\dot p - \dot{\hat p}+ W (x - \hat x)  + S^\top (u - \hat u) + A^\top (p - \hat p)\\
	\hat \sigma  + S(x - \hat x) + B^\top (p - \hat p) +  N_\U(u)
\end{array} \right).
\ed
Due to Assumption (A1), we have that $\dot{\hat x} , \, A, \, \dot{\hat p}, \, W, \, \hat \sigma \in L^\infty$, 
and $B, \, S \in W^{1,\infty}$. We remind that according to (\ref{EOS}) and (\ref{EOM}),  
$\hat u$ satisfies the inclusion $\hat \sigma + N_\U(\hat u) \ni 0$.

\bino
{\bf Step 2.} Define the map $\Lambda : L^1 \times \tilde Z \to L^\infty$ in the following way: for $u \in L^1$ and 
$z = (\xi,\pi,\rho) \in \tilde Z$,
\be \label{ELuz}
\Lambda(u,z) := \hat \sigma + S(x[u,z] - \hat x) + B^\top (p[u,z] - \hat p) - \rho,
\ee
where $(x[u,z], p[u,z])$ is the solution of the system 

\begin{align}\label{ELx}
\dot x &= \dot{\hat x} + A (x - \hat x)  + B (u - \hat u) - \xi, \quad x(0) = x^0, \\
\label{ELp}
-\dot p &= -\dot{\hat p} + W (x - \hat x)  + S^\top (u - \hat u) + A^\top (p - \hat p) - \pi, \quad p(T) = 0. 
\end{align}

Further we skip the argument $z$ if $z = 0$, so that $x[u] = x[u,0]$,  $p[u] = p[u,0]$, $\Lambda(u) := \Lambda(u,0)$. 

\begin{lem} \label{LEq_reduct}
	Strong bi-metric regularity of the set-valued map $L$
	at $\hat y$ for $0$ (in the spaces as in Theorem  \ref{TSbi-MR}) is equivalent to 
	strong bi-metric regularity of the map $\Lambda(\cdot,0) + N_\U(\cdot): \U \To L^\infty$ at $\hat u$ for zero,
	with disturbance space $W^{1,\infty} \st L^\infty$. 
\end{lem}

\begin{proof}{}  We shall prove that the bi-metric regularity of the map 
	$\Lambda(\cdot,0) + N_\U(\cdot)$ implies that of $L$, which will actually be used later. 
	The proof of the converse is similar and simpler.
	
	For any $z = (\xi,\pi,\rho) \in \tilde Z$ and $u \in L^\infty$, we have from (\ref{ELx}) that
	\be \label{EH61}
	\| l^x(\xi) \|_{1,\infty} \leq c_1 \| \xi \|_\infty, \quad
	\| l^x(\xi) \|_{1,1} \leq c'_1 \| z \|_Z,
	\ee
	where $l^x:L^\infty\to W^{1,\infty}$ is the linear map given by $l^x(\xi):=x[u,z] - x[u,0]$ , and $c_1$ and $c_1'$ are independent of $u$ and $z$.
	Using this and (\ref{ELp}), we obtain (also in a standard way) that
	\be \label{EH62}
	 \| l^p(\xi,\pi) \|_{1,\infty} \leq c_2 ( \| \xi \|_\infty + \| \pi \|_\infty),
	\quad \| l^p(\xi,\pi) \|_{1,1} \leq c'_2 \| z \|_Z,
	\ee  
	where $l^p:L^\infty \times L^\infty\to W^{1,\infty}$ is the linear map given by $ l^p(\xi,\pi):=p[u,z] - p[u,0]$, and $c_2$ and 
	$c'_2$ are constants such as $c_1$ and $c_1'$. Notice that the second inequalities in (\ref{EH61}) and (\ref{EH62}) imply that for a.e. $t \in [0,T]$
	\bd
	\max \{ |x[u,z](t)|, \,   |\dot x[u,z](t)|, \,  |p[u,z](t)|, \, |\dot p[u,z](t)| \} \leq M + c''( \| \xi \|_\infty + \| \pi \|_\infty),
	\ed
	where $c''$ is a constant. This will be used later to ensure that the appearing triples $(u, x[u,z], p[u,z])$ belong to the space $Y$.
	
	We may represent
	\bd
	\Lambda(u,z) = \Lambda(u) + Q(z),    
	\ed
	where 
	\bd
	Q(z) = S l^x(\xi) + B^\top l^p(\xi,\pi) - \rho, \quad \| Q(z) \|_{1,\infty} \leq c_3 \| z \|_\sim
	\ed
	is a linear map and $c_3$ is a constant.
	
	The inclusion $L(y) \ni z$ can be equivalently reformulated as
	\be \label{EH77}
	x = x[u,z], \quad p = p[u,z], \quad \Lambda(u,z) + N_\U(u) \ni 0.
	\ee
	In view of the obtained representations, the last relations are equivalent to
	\bd
	x = x[u] + l^x(\xi), \quad p = p[u] + l^p(\xi,\eta), \quad \Lambda(u) + Q(z) + N_\U(u) \ni 0.
	\ed
	Having in mind the estimations for  $\| l^x(\xi) \|_{1,\infty}$, $\| l^p(\xi,\pi) \|_{1,\infty}$ and $\| Q(z) \|_{1,\infty}$,
	obtaining Sbi-MR of $L$ from that of $\Lambda + N_\U$ becomes a routine task. 
	We will sketch the rest of the proof for completeness.
	
	First we observe that there is a constant $c_4$ such that $\|Q(z)\|_\infty \leq c_4 \| z \|_Z$.
	Let $\kappa$, $\al$ and $\beta$ be the constants in the definition of the Sbi-MR of the map $\Lambda + N_\U$.
	Fix
	\bd
	\bar \al = (c_1'+c_2')\bar{\beta}+\alpha, \qquad \bar \beta = \min\Big\{\frac{\beta}{c_3}, \frac{\bar M - M}{c''} \Big\},
	\qquad \bar \kappa = c'_1 + c'_2 + c_4 \kappa.
	\ed
	For any $z \in \tilde Z$ with $\| z \|_\sim \leq \bar \beta$ we have $\| Q(z) \|_{1,\infty} \leq \beta$.
	Then there exists a unique solution $u(z) \in \B_{L^1}(\hat u;\al)$ of the inclusion $\Lambda(u,z) + N_\U(u) \ni 0$.
	Moreover, for $z_1, \, z_2 \in \tilde Z$ with $\| z_i \|_\sim \leq \bar \beta$ we have 
	\bd
	\| u(z_1) - u(z_2) \|_1 \leq \kappa \| Q(z_1 - z_2) \|_\infty \leq c_4 \kappa \| z_1 - z_2 \|_Z. 
	\ed   
	From the first two relations in (\ref{EH77}) we have for $x(z_i) = x[u(z_i),z_i]$ and $p(z_i) = p[u(z_i),z_i]$
	\bd
	\| x(z_1) - x(z_2) \|_{1,1} + \| p(z_1) - p(z_2) \|_{1,1} \leq c'_1 \| z_1 - z_2 \|_Z + c'_2 \| z_1 - z_2 \|_Z.
	\ed 
	Thus $L$ is Sbi-MR at $\hat y$ for zero with constants $\bar \kappa, \, \bar \al, \, \bar \beta$.
\end{proof}

\bino
{\bf Step 3.} According to Lemma \ref{LEq_reduct}, it is enough to prove
Sbi-MR of $\Lambda + N_\U$ in the spaces specified in the formulation of the lemma.
It is convenient to use the notation 
\bd
\langle v, u \rangle_1 := \int_0^T \langle v(t), u(t) \rangle \dd t
\ed
for the duality pairing of $L^1$ and $L^\infty$, where $v \in L^\infty$ and $u \in L^1$.
The map $\Lambda : L^1 \to L^\infty$ is linear and continuous, and we shall show that its derivative, $\Lambda'$,
satisfies the equality
\be \label{ERep_der}
\langle \Lambda' \deu , \deu \rangle_1 = \Gamma(\deu), \quad \forall \, \deu \in L^1,
\ee
where the mapping $\Gamma: L^1 \to \R$ is defined in (\ref{EGamma}).\footnote%
{\label{FnDual} Similar representations are known, see e.g. in \cite{Hager-90}, but in the space $L^2$. 
	Here the space setting is different and the specificity of the affine problem is essential.}
In the notations introduced in this section the definition of $\Gamma$ reads as 
\be \label{EGamma=}
\Gamma(\deu) = \langle W \dex , \dex \rangle_1 + 2 \langle S \dex, \deu \rangle_1,
\ee
where $\dex$ is the solution of $\dot{\dex} = A \dex + B \deu$ with $\dex(0) = 0$. Let $\dep$ be the solution of 
the equation 
\bd
-\dot{\dep} = A^\top \dep + W \dex + S^\top \deu, \quad \dep(T) = 0.
\ed
Since  $u \mapsto \Lambda(u) := \hat \sigma + S(x[u] - \hat x) + B^\top (p[u] - \hat p)$ is linear, we deduce 
\be \label{lambdaprime} 
\Lambda ' (u) \deu = S \dex + B ^ \top \dep. 
\ee 

Integrating by parts the expression $\langle \dep, \dot{\dex} \rangle_1$ we obtain the equality
\bd
\langle \dep , A \dex + B \deu \rangle_1 =  \langle \dep, \dot{\dex} \rangle_1 =  
- \langle \dex, \dot{\dep} \rangle_1  = \langle \dex,  A^\top \dep + W \dex + S^\top \deu \rangle_1.
\ed
Hence,
\bd
\langle \dep , B \deu \rangle_1 = \langle \dex,  W \dex + S^\top \deu \rangle_1,
\ed
\bd
\langle B^\top \dep , \deu \rangle_1 = \langle W \dex,  \dex \rangle_1 + \langle S \dex , \deu \rangle_1,
\ed
\bd
\langle S \dex , \deu \rangle_1 + \langle B^\top \dep , \deu \rangle_1 = 
\langle W \dex,  \dex \rangle_1 + 2\langle S \dex , \deu \rangle_1 = \Gamma(\deu),
\ed
which implies (\ref{ERep_der})  in view of (\ref{lambdaprime}).

Equality (\ref{ERep_der})  allows to reformulate the inequality in Assumption (A2) as
\be \label{ERefA2}
\int_0^T \lll \sigma(t), \deu(t) \rrr  \dd t + \langle \Lambda' \deu, \deu \rangle_1 \geq c_0 \| \deu \|_1^2
\ee    
with $\sigma$ and $\deu$ as in (A2).

\bino
{\bf Step 4.} Next, we will prove that for every $\al \in (0, \al_0)$ (see Assumption (A2)) 
and for every $\Delta \in W^{1,\infty}$ with $\| \Delta \|_{1,\infty} < c_0 \al$ the inclusion 
\be \label{ELa_Del}
\Lambda(u) + N_\U(u) \ni \Delta 
\ee 
has a solution $\tilde u \in L^1$ satisfying $\| \tilde u - \hat u \|_1 < \al$.
For this, we consider the inclusion
\be \label{EH754}
\Lambda(u) + N_{\U \cap \B_{L^1}(\hat u; \al)}(u) \ni \Delta.
\ee
This inclusion represents the standard  necessary optimality condition for the problem 
\bd
\min \left\{J_0(u) := \int_0^T \Big[ \frac{1}{2}\langle W x[u],x[u] \rangle + 
\langle S x[u], u \rangle + \langle \Delta, u \rangle \Big] \right\},
\ed 
where $x[u]$ is defined around (\ref{ELx}), with the control constraints
$u \in \U$ and $u \in  \B_{L^1}(\hat u; \al)$. This is due to the well-known fact that $\Lambda(u)$ 
is the derivative of $J_0$ at $u$ in $L^1$ (the proof of this fact uses a similar argument as the 
proof of the relation (\ref{ERep_der})). Due to the weak compactness of $\U \cap \B_{L^1}(\hat u;\al)$
in $L^1$, this problem has a solution $\tilde u$, which then is a solution of (\ref{EH754}).

Now we use the relation 
\be \label{EdecomN}
N_{\U \cap \B_{L^1}(\hat u; \al)}(u) = N_{\U}(u) + N_{\B_{L^1}(\hat u; \al)}(u).
\ee
It follows from \cite[Theorem 3.1]{Burachik-04}, which, formulated for the 
particular space setting and sets, $\U \st L^1$ and $\V :=\B_{L^1}(\hat u;\al) \st L^1$, reads as follows:
the equality (\ref{EdecomN}) holds, provided that the set $\Epi s_\U + \Epi s_\V$ is weak${}^*$ closed,
where $\Epi s_\W$ is the epigraph of $s_\W$ and $s_\W : L^\infty \to \R$ is the support function 
to the set $\W \st L^1$, that is, $s_\W(l) := \sup_{w \in \W} \langle l, w \rangle_1$.  The week${}^*$ closedness
of this set is proved in Proposition 3.1, case (i), in \cite{Burachik-04}, which requires (in our case) that  
$\U$ and the interior of $\B_{L^1}(\hat u; \al)$ have a nonempty intersection, which is obviously fulfilled.

Due to (\ref{EdecomN}) and (\ref{EH754}), there exists $\nu \in N_{\B_{L^1}(\hat u; \al)}(\tilde u)$ such that 
\bd 
\nu + \Lambda(\tilde u) - \Delta \in - N_{\U}(\tilde u), 
\ed 
hence,
\bd 
\langle \nu, \hat u - \tilde u  \rangle_1 +  \langle \Lambda(\tilde u) - \Delta , \hat u - \tilde u \rangle_1 \geq 0.
\ed 
We have $ \langle \nu, \hat u - \tilde u  \rangle_1 \leq 0$ since $\hat u \in \B_{L^1}(\hat u; \al)$. Thus
\bd
\langle \Lambda(\tilde u), \hat u - \tilde u \rangle_1  - \langle \Delta , \hat u - \tilde u \rangle_1\geq 0.
\ed
Since $\Lambda$ is linear and satisfies (\ref{ERep_der}), and since $\Lambda(\hat u) = \hat \sigma$ in view of
(\ref{ELuz}), we obtain that
\bda
0 &\geq& \langle \Lambda(\hat u), \tilde u - \hat u \rangle_1 + 
\langle \Lambda(\tilde u) - \Lambda(\hat u), \tilde u - \hat u \rangle_1 
+ \langle \Delta , \hat u - \tilde u \rangle_1 \\ 
&=& \langle \Lambda(\hat u), \tilde u - \hat u \rangle_1 
+  \langle \Lambda'(\tilde u - \hat u), \tilde u - \hat u \rangle_1 + \langle \Delta , \hat u - \tilde u \rangle_1 \\
&=& \langle \hat \sigma, \tilde u - \hat u \rangle_1 
+  \Gamma(\tilde u - \hat u) + \langle \Delta , \hat u - \tilde u \rangle_1 .
\eda
Moreover, we have $\hat \sigma \in - N_\U(\hat u)$. 
Then Assumption (A2) in the form of (\ref{ERefA2}) applied for $\deu = \tilde u - \hat u$ and 
$\sigma = \hat \sigma$ implies that
\bd
0 \geq c_0 \| \tilde u - \hat u \|_1^2 + \langle \Delta , \hat u - \tilde u \rangle_1.
\ed
Hence,
\bd
\| \tilde u - \hat u \|_1 \leq \frac{\| \Delta \|_\infty}{c_0} < \al.
\ed
Since $\tilde u$ belongs to the interior of $\B_{L^1}(\hat u;\al)$, thus $N_{\B_{L^1}(\hat u;\al)}(\tilde u) = \{ 0 \}$, 
we obtain that $\nu = 0$, therefore $\tilde u$ is a solution of the inclusion (\ref{ELa_Del}). 

\bino
{\bf Step 5.} First, we shall estimate $\| \Lambda (u_1) - \Lambda(u_2) \|_{1,\infty}$ for two functions
$u_1, \, u_2 \in L^1$. Denote $\deu = u_1 - u_2$, $\dex = x[u_1] - x[u_2]$, $\dep = p[u_1] - p[u_2]$.
Then there is a constant $c_1$ independent of $u_1$ and $u_2$ such that 
\bd
\| \dex \|_\infty \leq c_1 \| \deu \|_1, \qquad \| \dep \|_\infty \leq c_1 \| \deu \|_1.
\ed
Using the definition of $\Lambda$ and Assumption (A1) we can estimate
\bd
\| \Lambda (u_1) - \Lambda(u_2) \|_{\infty} \leq c_2 \| \deu \|_1
\ed
with some constant $c_2$. Then

\begin{align*}
\Big\| \frac{\dd}{\dd t} (\Lambda(u_1) - \Lambda(u_2)) \Big \|_\infty&\le\| S(A\dex + B \deu) - B^\top (W \dex + S^\top \deu + A^\top \dep)  \|_\infty\\
&\quad+\| \dot S \dex + \dot B^\top \dep \|_\infty\\
&\le c_3 \| \deu \|_1,
\end{align*}

where $c_3$ is another constant and in the last estimate we use the assumed symmetry of 
$S B = \hat H_{ux} \hat B$. Thus
\be \label{EH12}
\| \Lambda (u_1) - \Lambda(u_2) \|_{1,\infty} \leq (c_2 + c_3) \| u_1 - u_2 \|_1 =: c_4 \| u_1 - u_2 \|_1.      
\ee 

Now we choose the number $\al$ in such a way that 
\bd
0 < \al \leq \al_0, \qquad c_0 \al \leq \al_0, \qquad (c_0 + c_4) \al \leq \gamma_0. 
\ed
Consider two disturbances $\Delta_1, \,\Delta_2 \in W^{1,\infty}$
with $\| \Delta_i \|_{1,\infty} < c_0 \al$, and two solutions $u_1, \, u_2 \in \B_{L^1}(\hat u;\al)$ of   
(\ref{ELa_Del}) corresponding to $\Delta_1$ and $\Delta_2$, respectively.

Let $\sigma := \Lambda(u_2) - \Delta_2$, by (\ref{EH12}) we have
\begin{align*}
\| \sigma - \hat \sigma \|_{1,\infty} &\leq \| \Lambda(u_2) - \Lambda(\hat u) \|_{1,\infty} + 
\| \Delta_2 \|_{1,\infty} \\
&\leq c_4 \| u_2 - \hat u \|_1 + c_0 \al\\
& < c_4 \al + c_0 \al \\
&\leq \gamma_0.
\end{align*}
Moreover, we have $\sigma = \Lambda(u_2) - \Delta_2 \in -N_\U(u_2)$ because 
$u_2$ solves the variational inequality (\ref{ELa_Del}) with $\Delta = \Delta_2$. Similarly as in Step 4 we obtain the following chain of inequalities:

\begin{align*}
0 &\geq\langle \Lambda(u_1) - \Delta_1, u_1 - u_2 \rangle_1 \\
&= \langle \Lambda(u_2) - \Delta_2, u_1 - u_2 \rangle_1 + 
\langle \Lambda(u_1) - \Lambda(u_2), u_1 - u_2 \rangle_1 + 
\langle  \Delta_2 - \Delta_1, u_1 - u_2 \rangle_1  \\
&= \langle \sigma, u_1 - u_2 \rangle_1 +
\langle \Lambda'(u_1 - u_2), u_1 - u_2 \rangle_1 + \langle  \Delta_2 - \Delta_1, u_1 - u_2 \rangle_1 \\
 &= \langle \sigma, u_1 - u_2 \rangle_1 +
\Gamma(u_1 - u_2) + \langle  \Delta_2 - \Delta_1, u_1 - u_2 \rangle_1,
\end{align*}
Having in mind also that $\| u_2 - \hat u\|_1 < \al \leq \al_0$, we can apply
Assumption (A2) (in the form as in (\ref{ERefA2})) to the latter inequality. We obtain 
\bd
0 \geq c_0\| u_1 - u_2 \|^2_1 + \langle \Delta_2 - \Delta_1, u_1 - u_2 \rangle_1,
\ed
which implies that $\| u_1 - u_2 \|_1 \leq \frac{1}{c_0} \| \Delta_1 - \Delta_2 \|_\infty$.
This proves the Sbi-MR property of $\Lambda + N_\U$ with constants $\kappa = (c_0)^{-1}$,
$\al$, and $\beta = c_0 \al$. The proof of Theorem \ref{TSbi-MR} is complete.  

\section{Some special cases} \label{Sspec}

We begin with few comments.  
Assumption (A2) with the particular choice $\sigma = \hat \sigma$, reads as
\be \label{EOV-cond}
\int_0^T \lll \hat \sigma(t), \deu(t) \rrr \dd t +  \Gamma(\deu) \geq c_0 \| \deu \|_1^2.
\ee
This inequality, required for all $\deu \in \U - \hat u$, is shown in \cite{Osmol+Vel-19} to be sufficient 
for the property of {\em strong metric sub-regularity}, which is substantially weaker than Sbi-MR.
Moreover, the condition\footnote{\label{FnTaylor} The left-hand side in the next inequality is just the second order Taylor
	expansion of the objective functional $J(u)$ in (\ref{Eg}).}
\bd 
\int_0^T \lll \hat \sigma(t), \deu(t) \rrr \dd t + \frac{1}{2} \Gamma(\deu) \geq c_0 \| \deu \|_1^2,
\quad \forall \,  \deu \in \U - \hat u, \quad \| \deu \|_1 \;  \mbox{ -- small enough},
\ed
is sufficient for strict local optimality of $\hat u$ in an $L^1$-neighborhood. 
This last condition is weaker than (\ref{EOV-cond}), as shown in \cite{Osmol+Vel-19}.

Assumption (A2) is fulfillled on the following (more compact) one. 

\bino
{\em Assumption (A2').}
There exist numbers $c_0$, $\al_0 > 0$ and $\gamma_0 > 0$ such that
\be \label{EA2'}
\int_0^T |\lll \sigma(t), \deu(t) \rrr |  \dd t +  \Gamma(\deu) \geq c_0 \| \deu \|_1^2,
\ee 
for every function $\sigma \in \B_{W^{1,\infty}}(\hat \sigma; \gamma_0)$ 
and for every $\deu \in \U - \U$ with $\| \deu \|_1 \leq \al_0$.

\bino
Obviously (A2') implies (A2), since for $\sigma \in - N_\U(u)$ and $u' \in \U$ it holds that $\lll \sigma(t), u'(t) - u(t) \rrr \geq 0$. 

\bino
Now we focus on the first-order term in (\ref{EA2'}) under an additional condition introduced in \cite{Felge2003}
in a somewhat stronger form and for box-like sets $U$.

\bino
{\em Assumption (B).} The set $U$ is a convex and compact polyhedron. Moreover, there exist numbers 
$\kappa > 0$ and $\tau > 0$ such that for every unit vector $e$ parallel to some edge of $U$ and for every $s \in [0,T]$
for which $\lll \hat \sigma(s), e \rrr = 0$ it holds that
\bd
| \lll \hat \sigma(t), e \rrr | \geq \kappa | t - s |   \qquad t \in [ s - \tau, s + \tau] \cap [0,T].
\ed

The next lemma claims that Assumption (B) remains valid for all functions $\sigma$ close enough to 
$\hat \sigma$ in $W^{1,\infty}$. 

\begin{lem} \label{LB}
	Let assumptions (A1) and (B) be fulfilled. Then  there exist numbers 
	$\kappa' > 0$, $\tau' > 0$ and $\gamma' > 0$ such that for every function 
	$\sigma \in \B_{W^{1,\infty}}(\hat \sigma; \gamma')$, for every unit vector $e$ parallel to some edge of $U$ 
	and for every $s \in [0,T]$ for which $\lll \sigma(s), e \rrr = 0$, it holds that
	\bd
	| \lll \sigma(t), e \rrr | \geq \kappa' | t - s |   \qquad t \in [ s - \tau', s + \tau'] \cap [0,T].
	\ed
\end{lem}

\begin{proof}{}
	The proof combines arguments from the proof of Proposition 3.4 in \cite{QSV-18} and 
	the proof of Proposition 4.1, therefore we only sketch it focusing on the differences
	with the proofs mentioned above.
	
	First of all, Assumption (B) implies that the reference control $\hat u$ is piece-wise constant. This follows from the fact that
	$\lll \hat \sigma(t), e \rrr$ has not more than $T/\tau +1$ zeros in $[0,T]$ and $U$ has a finite number of edges. 
	More details are given in the proof of Proposition 4.1 in \cite{Osmol+Vel-19}. 
	
	From the definition of $\hat \sigma$, (A1) and the fact that $\hat u$ is a piece-wise constant function we obtain that 
	$\hat \sigma$ has a piece-wise continuous derivative. Let us fix $e$ as in Assumption (B), and denote 
	$\hat \sigma_e := \lll \hat \sigma(t), e \rrr$. Let $\hat s_1, \ldots, \hat s_k$ be the zeros of $\hat \sigma_e$ in 
	$[0,T]$. For $\delta > 0$ define
	\bd
	\Omega(\delta) := \cup_{i=1}^{k} [\hat s_i - \delta, \hat s_i + \delta].
	\ed
	Choose $\delta > 0$ so small that $\delta < \tau$ and there are no other points of discontinuity of $\dot {\hat \sigma}$
	in $\Omega(\delta)$ except possibly $\hat s_1, \ldots, \hat s_k$. Denote
	\bd
	\dot{\hat \sigma}_e^{-}(\hat s_i) := \lim_{t \rightarrow \hat s_i -0} \; \dot {\hat \sigma}(t),  \quad
	\dot{\hat \sigma}_e^{+}(\hat s_i) := \lim_{t \rightarrow \hat s_i +0} \; \dot {\hat \sigma}(t), \quad i=1, \ldots, k.
	\ed
	By choosing $\delta > 0$ smaller, if needed, we may ensure that 
	\bd
	| \dot {\hat \sigma}(t) -  \dot{\hat \sigma}_e^{-}(\hat s_i) | \leq \frac{\kappa}{4} 
	\;\;\mbox{for} \; t \in [\hat s_i - \delta, \hat s_i],   \quad 
	| \dot {\hat \sigma}(t) -  \dot{\hat \sigma}_e^{+}(\hat s_i) | \leq \frac{\kappa}{4} 
	\;\; \mbox{for} \;t \in [\hat s_i, \hat s_i + \delta].
	\ed
	Then for every $i$ and $t \in [\hat s_i-\delta, \hat s_i ]$ we have from Assumption (B) that
	\bda
	\kappa |t - \hat s_i| &\leq& |\hat \sigma_e(t) - \hat \sigma_e(\hat s_i)| \\
	&=&
	\Big| \int_{\hat s_i}^t  \dot{\hat \sigma}_e(\theta) \dd \theta\Big| \\
	&\leq&
	\int_{\hat s_i}^t  |\dot{\hat \sigma}_e^{-}(\hat s_i) | \dd \theta + 
	\int_{\hat s_i}^t |\dot{\hat \sigma}_e^{-}(\hat s_i) - \dot{\hat \sigma}_e(\theta)| \dd \theta \\
	& \leq& |t - \hat s_i | \, |\dot{\hat \sigma}_e^{-}(\hat s_i)|  + \frac{\kappa}{4}  |t - \hat s_i |
	\eda
	Hence,
	\bd
	|\dot{\hat \sigma}_e^{-}(\hat s_i)| \geq \frac{3 \kappa}{4}.
	\ed
	Analogously we obtain the same estimate for $|\dot{\hat \sigma}_e^{+}(\hat s_i)|$.
	
	Obviously there exists $\eta > 0$ such that
	$| \hat \sigma_e(t) | \geq \eta$ for every $t \in [0,T] \sm \Omega(\delta/2)$. 
	By choosing the number $\gamma \in (0,\kappa/4]$ sufficiently small we have that 
	for every $\sigma \in \B_{W^{1,\infty}}(\hat \sigma;\gamma)$ the function $\sigma_e = \lll \sigma, e \rrr$ has no
	zeros in $ [0,T] \sm \Omega(\delta/2)$. Now let us take an arbitrary $\sigma$ as in the last sentence.
	Let $s$ be an arbitrary zero of $\sigma_e$ in $[0,T]$. Then there exists $\hat s_i$ such that $|s - \hat s_i| \leq \delta/2$.
	For $t \in [ s-\delta/2, s+ \delta/2 ]$ we can estimate
	\begin{align*}
	|\sigma_e(t) | &= \Big| \int_s^t \dot \sigma_e(\theta) \dd \theta \Big|\\
	& \geq \Big| \int_s^t \dot {\hat\sigma}_e(\theta) \dd \theta \Big| -   
	\int_s^t |\dot \sigma_e(\theta) - \dot {\hat\sigma}_e(\theta) | \dd \theta \Big| \\
	&\geq   \Big| \int_s^t \dot {\hat\sigma}_e(\theta) \dd \theta \Big| - \gamma |t-s|.
	\end{align*}
	For the last integral we have
	\bd
	\Big| \int_s^t \dot {\hat\sigma}_e(\theta) \dd \theta \Big| \geq  \Big| \int_s^t  \zeta(\theta) \dd \theta \Big| - 
	\int_s^t |\dot {\hat\sigma}_e(\theta) - \zeta(\theta)| \dd \theta,
	\ed
	where $\zeta(\theta)$ is either $\dot{\hat \sigma}_e^{-}(\hat s_i)$ or $\dot{\hat \sigma}_e^{+}(\hat s_i)$
	depending on whether $\theta < \hat s_i$ or $\theta > \hat s_i$. Thus we can estimate 
	\bd
	|\sigma_e(t) | \geq \frac{3 \kappa}{4} |t-s| - \frac{\kappa}{4} |t-s|  - \gamma |t-s|
	\geq  \frac{\kappa}{4} |t-s|. 
	\ed
	Thus we obtain the claim of the lemma with $\kappa' = \kappa/4$, $\tau' = \delta/2$ and $\gamma' = \gamma$. 
\end{proof}

\begin{prop} \label{PB}
	Let assumptions (A1) and (B) be fulfilled. Then there exist numbers $c_0$, $\al_0 > 0$ and $\gamma_0 > 0$ such that
	\be \label{EB}
	\int_0^T |\lll \sigma(t), \deu(t) \rrr |  \dd t  \geq c_0 \| \deu \|_1^2,
	\ee 
	for every function $\sigma \in \B_{W^{1,\infty}}(\hat \sigma; \gamma_0)$ 
	and for every $\deu \in \U - \U$ with $\| \deu \|_1 \leq \al_0$.
\end{prop}

\bino
Having at hand Lemma \ref{LB}, the proof repeats that of  Proposition 4.1 in \cite{Osmol+Vel-19}.

\bino
\begin{rem} \label{Rkappa} {\em
		A more slightly precise modification of the proof of Lemma \ref{LB} shows that the number $\kappa'$ can be taken 
		as any number smaller than $\kappa$ (from Assumption (B)).
		Moreover, the constant $c_0$ in Proposition \ref{PB} is directly related 
		with number $\kappa'$ (thus with $\kappa$). In the simplest case of scalar control and $U = [u_1, u_2]$ 
		is straightforward. As obtained in the proof of Lemma \ref{LB}, Assumption (B) implies in this case that 
		$\hat \sigma$ has finite number of zeros, $\hat s_1, \ldots, \hat s_k$, and $\dot{\hat \sigma}$ is piece-wise continuous.
		If the number $Q$ satisfies 
		\bd
		\lmi_{t \rightarrow \hat s_i} |\dot{\hat \sigma}(t)| \geq Q \quad 1 = 1, \ldots, k,
		\ed
		(the $\lmi$ is taken over all $t$ at which the derivative exists) 
		then a simple calculation shows that the claim of Proposition \ref{PB} holds with any number 
		$c_0 \leq Q/(8 k(u_2 - u_1))$.
}\end{rem}

\bino
{\bf Example 1.} 
This example shows that Sbi-MR of the optimality mapping may hold even for problems that are non-convex,
namely, the objective functional $J$ in (\ref{Eg}) is even directionally non-convex at the optimal control $\hat u$.
Consider the problem 
\bd
\min \Big\{ J(u) := \int_0^1 \Big[ - \frac{\al}{2} (x(t))^2 - \beta x(t) + u(t) \Big] \dd t \Big\},
\ed
subject to 
\bd
\dot x = u, \quad x(0) = 0, \quad u(t) \in [0,1].
\ed
Here $\al$ and $\beta$ are positive parameters satisfying $\beta > 1$, $2 \al \leq \beta$.

The solution of the adjoint equation $\dot p = \al x + \beta$, $p(1) = 0$ is strictly monotone increasing and the 
switching function, $\sigma(t) = p(t) + 1$, is positive at $t= 1$. This implies that only optimal control 
has the structure  
\bd
\hat u(t) = \left\{ \begin{array}{cl}
	1 & \mbox{ for } t \in [0,\tau], \\
	0 & \mbox{ for } t \in (\tau, 1].
\end{array} \right. 
\ed 
The corresponding solutions of the primal and the adjoint equations are
\bd
\hat x(t) = \left\{ \begin{array}{cl}
	t & \mbox{ for } t \in [0, \tau], \\
	\tau & \mbox{ for } t \in (\tau, 1],
\end{array} \right.
\ed
and
\begin{align*}
\hat p(t) = \left\{ \begin{array}{cl}
\frac{\al}{2}(\tau^2 + t^2) + \beta t - \al \tau - \beta & \mbox{ for } t \in [0,\tau], \\
t (\al \tau + \beta) - \al \tau - \beta & \mbox{ for } t \in (\tau, 1].
\end{array} \right.
\end{align*}
A simple calculation shows that for $\beta > 1$ the optimal control $\hat u$ has
\bd
\tau = \frac{ -(\beta - \al) + \sqrt{(\beta - \al)^2 + 4 \al (\beta - 1) } }{2 \al} \, \in \, (0,1). 
\ed
For the corresponding switching function $\hat \sigma = \hat p + 1$ we have 
$\dot {\hat \sigma}(\tau) = \al \tau + \beta > \beta$. 
Then Assumption (B) is fulfilled with $\kappa < \beta$. According to Remark \ref{Rkappa}, 
we have 
\bd
\int_0^1 |\hat \sigma(t) \deu(t)| \dd t \geq \frac{\beta}{2} \| \deu \|^2_1 \qquad \forall \, \deu \in \U - \U
\;   \mbox{ with a sufficiently small $\| \deu \|_1$}.
\ed
Moreover, 
\bd
\Gamma(\deu) = -\int_0^1 \al (\dex(t))^2 \dd t = - \al \int_0^1 \Big(\int_0^t \deu(s) \dd s \Big)^2 \dd t
\geq - \al \| \deu \|^2_1.
\ed
Thus for $2 \al < \beta$ Assumption (A2') is fulfilled and the optimality mapping for the considered problem
is Sbi-MR at $(\hat x, \hat u, \hat p)$ for zero.
On the other hand, considering again the expression for the second variation $\Gamma$, we see that 
$\Gamma(\deu) < 0$,
except some specially constructed control variations $\deu$. Thus the objective functional $J(u)$    
in this example is not convex even directionally at the solution point $\hat u$.

\section{An application: uniform convergence of the Euler discretization} \label{Sunif_discr}

In this section we prove that the sufficient conditions for Sbi-MR given in Theorem \ref{TSbi-MR} imply a 
property that can be called {\em uniform strong sub-regularity} concerning a family of optimal control problems 
``neighboring'' a given reference problem. This property is shown to imply a {\em uniform} error estimate for the 
accuracy of the Euler discretization scheme, applied to any of the problems of the family.

We consider again the reference problem (\ref{Eg})-(\ref{Eu}) together with the fixed solution $(\hat x, \hat p, \hat u)$
of its optimality system (\ref{EPx})--(\ref{EPu}). The assumptions in Theorem \ref{TSbi-MR} will hold in the isection, 
with the additional supposition that $f$ and $g$ are time-invariant. 

Together with the reference problem, we consider a family of problems of the same kind, each defined by a pair
of time-invariant functions $\pi := (\tilde f, \tilde g)$ satisfying Assumption (A1) (with $f$ and $g$ replaced with $\tilde f$ 
and $\tilde g$). Any such pair will be called admissible, and $(\p_{\pi})$ will denote the problem corresponding 
to the pair $\pi$, that is, the problem
\be\label{Eg2}
\min_{u\in\U}\left\lbrace \int_{0}^{T}\tilde g(x(t),u(t)) \dd t \right\rbrace 
\ee
subject to 
\be\label{Es2}
\dot x(t)=\tilde f(x(t),u(t)),\hspace*{0.3cm}x(0)=x^0.
\ee
Due to relation (\ref{EM}), we restrict our consideration to admissible pairs $\pi$ defined on the set
$D:= \B_{\R^n}{(0,\bar M)} \times U$.
Given a positive number $\rho$, we denote by $\HH_\rho$ the set of all admissible pairs 
$\pi = (\tilde f, \tilde g)$ such that 
\be
\|\tilde f - f\|_{1,\infty} + \|\tilde g - g\|_{1,\infty}  \leq \rho,
\ee
where the $W^{1,\infty}$-norms are taken for functions defined on the set $D$.

For a given $\pi=(\tilde f,\tilde g)\in\mathcal H_\rho$, we consider the mapping $\Phi_{\pi}:Y\to Z$ defined by 
\be
\Phi_\pi(x,p,u)=\left( \ba{c}
\dot x - \tilde f(x,u) \\
\dot p +  \nabla_{\!\! x} \tilde H(x,p,u) \\
\nabla_{\!u} \tilde H(x,p,u) + N_\U(u) \ea \right) 
\ee
where $\tilde H$ is the Hamiltonian corresponding to the pair $\pi$, and
where as before $N_\U(u)\subset L^\infty$ is the normal cone to the set $\U$ of admissible controls at $u$.
The following  lemma is technical.

\begin{lem}\label{LSbi}
	Let $\pi=(\tilde f,\tilde g)$ belong to $\HH_\rho$ and  $\varphi_{\pi}:Y\to Z$ be defined as 
	\be\label{Emapphi}
	\varphi_{\pi}(x,p,u)=\left( \ba{c}
	\varphi^1_{\pi}(x,p,u) \\
	\varphi^2_{\pi}(x,p,u) \\
	\varphi^3_{\pi}(x,p,u)
	\ea 
	\right):=\left( \ba{c}
	f(x,u)-\tilde f(x,u) \\
	\nabla_{\!\! x}\tilde H(x,p,u)-\nabla_{\!\! x}  H(x,p,u) \\
	\nabla_{\!u}\tilde  H(x,p,u)-\nabla_{\!u} H(x,p,u)
	\ea \right).
	\ee
	There exists a positive constant $c$ such that 
	
	\be
	d_{Z}(\varphi_{\pi}(y),0)\le c \rho\hspace*{0.5cm}\forall y\in Y.
	\ee
\end{lem}

\begin{proof}{}
	Let $y=(x,p,u)\in Y$. We estimate each one 
	of the components of $\varphi_{\pi}(y)$. First, 
	
	\bd
	\|\varphi^1_{\pi}(y)\|_1=\| f(x,u)-\tilde f(x,u)\|_1\le T\rho.
	\ed
	In a similar way,
	\begin{align*}
	\|\varphi^2_{\pi}(y)\|_1&=\|\nabla_{\!\! x} \tilde H(x,p,u)-
	\nabla_{\!\! x}  H(x,p,u)\|_1\\
	&\le\|\tilde f_x- f_x\|_1\|p\|_\infty+\|\tilde g_x- g_x\|_1\\&\le(\bar M+1)T\rho.
	\end{align*}
	Analogously, 
	\bd
	\|\varphi^3_{\pi}(y)\|_\infty\le (\bar M+1)\rho	.
	\ed
	The result follows.
\end{proof}

We remind the notion of Strong Metric sub-Regularity (SMsR) for a set-valued mapping $\Phi:Y\to Z$. 
We make use of this notion in the following results.

\begin{definition}
	A set valued mapping $\Phi:Y\to Z$ is Strongly Metrically sub-Regular (SMsR) at $y^*$ 
	for zero if $0\in\Phi(y^*)$ and there exist $a,b>0$ and $\kappa>0$ such that for any 
	$z\in B_{Z}(0,b)$ and any solution $y\in B_Y(y^*,a)$ of the inclusion $z\in\Phi(y)$ it holds that 
	$d_Y(y,y^*)\le\kappa d_Z(z,0)$. We call $a,b$ and $\kappa$ the parameters of SMsR.
\end{definition} 

According to Theorem 3.1 in \cite{Osmol+Vel-19}, Assumption (A2) implies that the optimality map 
$F$ in (\ref{EOM}) is SMsR at $\hat y$ for zero (see Section \ref{Sspec}).
We fix its parameters $a,b>0$ and $\kappa>0$ of SMsR. 

\begin{prop}\label{Tbim}
	Let $\pi$ belong to $\mathcal H_\rho$ and $y^*\in B_Y(\hat y,a)$ be a solution of problem $(\p_\pi)$. 
	There exists a positive constant $\kappa'$ such that
	\be 
	d_Y(\hat y,y^*)\le \kappa'\rho,
	\ee
	for all sufficiently small $\rho$.
\end{prop}

\begin{proof}{}
	We can write $\Phi_\pi=\varphi_\pi+F$, where $\varphi_\pi$ is the map 
	$(\ref{Emapphi})$ in Lemma \ref{LSbi} and $F$ is the optimality mapping (\ref{EOM}).
	Let $c>0$ be the constant in that lemma, so that  $d_Z(\varphi_\pi(y),0)\le c\rho$ for all $y\in Y$. We can choose $\rho$ small enough to ensure $\varphi_{\pi}(y)\in B_Z(0,b)$ for all $y\in Y$. Since $y^*$ is a solution of problem $(\p_{\pi})$, the inclusion $0\in \varphi_\pi(y^*)+F(y^*)$ is satisfied. By SMsR, we have the desired inequality with $\kappa':=c\kappa$.
\end{proof}

Analogously as we defined the functional $\Gamma$;   given a $\pi\in \mathcal H$ 
and a reference solution $y^*$ of problem $(\p_\pi)$, we consider the functional 
$\Gamma_{\pi}:L^1\to\mathbb R$ defined in terms of $\pi$ and $y^*$ as in (\ref{EGamma}). Explicitly,

\bd
\Gamma_\pi(\deu) = \int_0^T\Big[\lll \tilde H_{xx}(y^*(t))\dex(t),\dex(t)\rangle+2\langle\,
\tilde H_{ux}(y^*(t))\dex(t),\deu(t)\rrr \Big] \dd t,
\ed
where $\dex$ is the solution of the equation $\dot{\dex}(t) =  \tilde f_x(x^*(t),u^*(t)) \dex(t) + \tilde f_u(x^*(t),u^*(t)) \deu(t)$ 
with initial condition $\dex(0) = 0$.

The following lemma establishes an estimation involving the functionals $\Gamma_{\pi}$ and $\Gamma$.
\begin{lem}\label{Lestgam}
	Let $\pi$ belong to $\mathcal H_\rho$ and $y^*\in B_Y(\hat y,a)$ be a solution of problem $(\p_\pi)$. There exists a constant $\eta>0$ such that 
	\bd
	|\Gamma(v-u^*
	)-\Gamma_{\pi}(v-u^*)|\le \eta \hspace*{0.05cm}\rho\|v-u^*\|_1^2\hspace*{0.5cm}\forall v\in \U,
	\ed
	for all sufficiently small $\rho$.
\end{lem}

\begin{proof}{}
	Using Proposition $\ref{Tbim}$ and the Lipschitz continuity of the functions involved, we can find positive 
	constants $c_w$ and $c_s$ such that
	\be\label{Edesi}
	\|\hat H_{xx}-\tilde H_{xx}^*\|_1\le c_w \rho,
	\ee 
	and
	\be
	\|\hat H_{ux}-\tilde H_{ux}^*\|_\infty\le c_s \rho.
	\ee
	Let $v\in\U$ and $v'=v-u^*$, we denote by $\delta\hat x$ and $\dex^*$ the solutions of 
	\be
	\dot x= \hat Ax+\hat Bv',\hspace*{0.2cm}x(0)=0,\hspace{1cm}
	\dot x= \tilde A^*x+\tilde B^*v',\hspace*{0.2cm}x(0)=0,
	\ee
	respectively. There exist positive constants $d_1$ and $d_2$ such that 
	\be
	\max\left\lbrace \|\delta \hat x\|_\infty,\|\dex^*\|_\infty\right\rbrace \le d_1\|v'\|_1,
	\ee
	and 
	\be\label{Edesf}
	\|\delta \hat x-\dex^*\|_\infty\le d_2\rho\|v'\|_1.
	\ee
	Now,
	\begin{align*}\displaystyle
	|\Gamma(v')-\Gamma_{\pi}(v')|&\le\left|\int_{0}^{T}\left[\langle\,\hat{H}_{xx}\delta \hat x,\delta \hat x\rangle-
	\langle\,\tilde H^*_{xx}\delta x^*,\delta x^*\rangle\right]\right|\\
	&\quad+2\left|\int_{0}^{T}
	\left[\langle\,\hat{H}_{ux}\delta \hat x-\tilde H^*_{ux}\delta x^*,v'\rangle\right]\right|\\
	&\le\int_{0}^{T}|\langle\,\hat{H}_{xx}\delta \hat x,\delta \hat x-\delta x^*\rangle|+
	\int_{0}^{T}|\langle\,\hat{H}_{xx}\delta\hat x-\tilde H^*_{xx}\delta x^*,\dex^*\rangle|\\
	&\quad+
	2\int_{0}^T|\langle\,\hat{H}_{ux}\delta \hat x- \tilde H^*_{ux}\dex^*,v'\rangle|\\
	&\le \|\hat{H}_{xx}\delta \hat x\|_1\|\delta \hat x-\dex^*\|_\infty+\big(\|\hat{H}_{xx}(\delta\hat x-\dex^*)\|_1\\
	&\quad+
	\|(\hat{H}_{xx}-\tilde H^*_{xx})\dex^*\|_1\big)\|\delta x^*\|_\infty\\
	&\quad+\big(\|\hat{H}_{ux}(\delta \hat x-\dex^*)\|_\infty+\|(\hat{H}_{ux}- \tilde H^*_{ux})\dex^*\|_{\infty}\big)\|v'\|_1.
	\end{align*}
	Taking (\ref{Edesi})-(\ref{Edesf}) into account, the result follows.
\end{proof}

\begin{thm}\label{Tstaa2}
	There exist $ \zeta, \tilde a,\tilde b>0$ and $\tilde\kappa>0$ such that if
	$\pi\in \mathcal H_\zeta$ and $y^*\in B_Y(\hat y,a)$ is a solution for  problem $(\p_\pi)$, then the map 
	$\Phi_\pi$ is SMsR at $y^*$ for zero with parameters $ \tilde a,\tilde b,\tilde\kappa$.
\end{thm}

\begin{proof}{}
	Let $c_0$, $\al_0 > 0$ and $\gamma_0 > 0$ be the numbers in Assumption (A2). 
	If $y^*$ is a solution for problem $(\p_\pi)$, we have $0\in\Phi_{\pi}(y^*)$. 
	By Proposition \ref{Tbim}, there exists $\zeta>0$ such that for any 
	$\pi\in \mathcal H_\zeta$,  $d_Y(\hat y,y^*)<\kappa'\zeta$ 
	for some constant $\kappa'>0$. We consider  $\zeta$  small enough to guarantee 
	$\|\hat \sigma-\tilde\sigma^*\|_{1,\infty}\le\gamma_0$ 
	(in the estimation of $\| \dot{\hat \sigma} -\dot{\tilde\sigma}^*\|_{\infty}$ we use the symmetry of $\hat H_{ux} \hat B$
	similarly as in the estimation before (\ref{EH12})) and $\|\hat u-u^*\|_1< \alpha_0/2$. 
	Let $\tilde \alpha_0:= \alpha_0/2$, so $B_{L^1}(u^*;\tilde \alpha_0)\subset B_{L^1}(\hat u;\alpha_0)$. 
	
	By Assumption (A2),
	\be
	\int_{0}^{T}\langle\,\tilde \sigma^*,v-u^*\rangle+\Gamma(v-u^*)\ge c_0\|v-u^*\|_1^2
	\hspace*{0.3cm}\forall{v\in\mathcal U\cap B_{L^1}(u^*;\tilde \alpha_0)},
	\ee
	or
	
	\bd
	\int_{0}^{T}\langle\,\tilde \sigma^*,v-u^*\rangle+\Gamma_\pi(v-u^*)\ge c_0\|v-u^*\|_1^2+
	\Big[\Gamma_\pi(v-u^*)-\Gamma(v-u^*)\Big],
	\ed
	for all ${v\in\mathcal U\cap B_{L^1}(u^*;\tilde \alpha_0)}$.
	Taking into account Lemma \ref{Lestgam}, we can choose $\zeta$ smaller if needed to ensure
	\bd 
	\Gamma_\pi(v-u^*)-\Gamma(v-u^*)\ge -\frac{c_0}{2}\|v-u^*\|_1^2\hspace*{0.5cm}
	\forall{v\in\mathcal U\cap B_{L^1}(u^*;\tilde\alpha_0)}.
	\ed
	Thus,
	\be\label{Eprop}
	\int_{0}^{T}\langle\,\sigma^*,v-u^*\rangle+\Gamma_{\pi}(v-u^*)\ge
	\frac{c_0}{2}\|v-u^*\|_1^2\hspace*{0.5cm}\forall{v\in\mathcal U\cap B_{L^1}(u^*;\tilde\alpha_0)}.
	\ee
	Let $L$ be a bound for the Lipschitz constants of $f,g$ and $H$ their first derivatives in $x$, 
	and $H_{xu},H_{up}$. It is easy to see that $\tilde L:=L+2(1+\bar M)\zeta$ is a bound for the Lipschitz constants 
	of $\tilde f,\tilde g$ and $\tilde H$, their first derivatives in $x$, and $\tilde H_{xu},\tilde H_{up}$ for all 
	$\pi\in \mathcal H_\zeta$. Analogously, we can find a bound $\tilde M$, depending on $\zeta$ and $\bar M$, 
	for the functions $\tilde f, \tilde H$ and their derivatives (see Remark 2.1 in \cite{Osmol+Vel-19}).
	Finally, (\ref{Eprop}) implies that the hypotheses of Theorem 3.1 in \cite{Osmol+Vel-19} are fulfilled. We conclude that $\Phi_\pi$ is SMsR at $y^*$ for zero. 
	According to that theorem, the parameters of SMsR can be chosen as depending only on 
	$\tilde M, T, c_0$ and $\tilde L$. This completes the proof.
\end{proof}

From now on, we only consider elements $\pi\in \mathcal H_\zeta$, since the latter theorem 
ensures that each map $\Phi_\pi$ is SMsR at a solution $y^*\in B_Y(\hat y,a)$ for zero. 
This automatically ensures that $y^*$ is the unique local  solution in $B_Y(\hat y,a)$  of the inclusion $0\in\Phi_\pi(y)$.

Let $\left\lbrace t_n\right\rbrace_{n=0}^{N}$ be a grid on $[0,T]$ with equally spaced nodes and 
a step size $h$, that is, $t_k=k T/N$ for $i=0,\dots, N$. Given a $\pi\in \mathcal H_\zeta$, 
the discrete time problem $(\mathcal{P}_{\pi}^h)$ obtained by the Euler discretization is

\be
\displaystyle\min_{u\in U^N}\left[ h\sum_{i=0}^{N-1}\tilde g(x_i,u_i)\right]
\ee
subject to
\be
x_{i+1}=x_i+h\tilde f(x_i,u_i),\hspace*{0.3cm}x_0=x^0.
\ee
The local form of the discrete time minimum principle implies that for any locally optimal solution $(x,u)$ 
of problem ($\mathcal{P}_{\pi}^h$) there exists a vector $p=(p_0,\dots,p_N)$ such that
\begin{align}\label{Edmp1}
x_{i+1}&=x_i+h\tilde f(x_i,u_i), \hspace*{0.1cm}x_0=x^0,\\
\lambda_i&=\lambda_{i+1}+h\nabla_{\!\! x}\tilde H(x_i,u_i,p_{i+1}), \hspace*{0.1cm} p_N=0,\\
0&\in\nabla_{\!u}\tilde H(x_i,u_i,p_{i+1})+N_{U}(u_i),\label{Edmp2}
\end{align}
where $i$ runs between $0$ and $N-1$. Let $(x^h,u^h)$ be a solution of problem ($\mathcal{P}_{\pi}^h$) 
and $p^h$ the corresponding co-state vector, so that $y^h=(x^h,p^h,u^h)$ satisfies (\ref{Edmp1})-(\ref{Edmp2}). 
In order to compare this solution with the reference solution of $y^*=(x^*,p^*,u^*)$ of the continuous-time problem $(\p_\pi)$,
we embed the sequence $(x^h,p^h,u^h)$ into the space $W^{1,1}\times W^{1,1}\times L^1$ 
considering $y_h=(x_h,p_h,u_h)$ defined by
\be\label{Eembedded}
x_h(t):=x_i^h+\frac{t-t_i}{h}(x^h_{i+1}-x^h_i),\hspace*{0.3cm}u_h(t):=u_i^h,\hspace*{0.3cm}
p_h(t):=p_i^h+\frac{t-t_i}{h}(p^h_{i+1}-p^h_i),
\ee
for $t\in[t_i,t_{i+1})$, $i=0,\dots, N-1$.

\bino
We need the following technical assumption to apply results in \cite{Osmol+Vel-19}. It is a crucial assumption, 
at least because it may happen that $y_h$ is close to some other local solution of the continuous-time problem, 
and we have to eliminate this possibility.

\bino
{\em Assumption (C1).} Let $\pi\in\mathcal H_\zeta$. We assume that problem $(\p_\pi)$ has 
a solution $y^*$ in $B_Y(\hat y,a)$. Moreover, the embedded solution $y_h$ in (\ref{Eembedded}) 
of problem $(\mathcal P_\pi^h)$ belongs to $B_Y(y^*,\tilde a)$  for all sufficiently small $h$.

\bino
The following theorem is a direct consequence of Theorem \ref{Tstaa2} and Theorem 5.1 in \cite{Osmol+Vel-19}.

\begin{thm}
	There exists a positive constant $C$ such that for 
	all $\pi\in \mathcal H_\zeta$ for which Assumption (C1) holds, the estimate
	\be\label{Eeuler}
	\|x_h-x^*\|_{1,1}+\|p_h-p^*\|_{1,1}+\|u_h-u^*\|_1\le Ch
	\ee
	holds for all sufficiently small $h$.
\end{thm}

\begin{proof}{}
	By Theorem \ref{Tstaa2}, the parameters $\tilde a,\tilde b,\tilde \kappa$ of SMsR of $\Phi_\pi$ at $y^*$ 
	for zero are the same for all $\pi\in \mathcal H_\zeta$ satisfying Assumption (C1).
	
	Let $\pi\in\mathcal H_\zeta$. In order to make use of the SMsR property of the map $\Phi_\pi$, 
	we have to estimate the residuals 	
	\begin{align*}
	\Delta_1&:=\dot x_h-\tilde f(x_h,u_h),\\
	\Delta_2&:= \dot p_h+\nabla_x \tilde H(x_h,p_h,u_h),\\
	\Delta_3&:=\nabla_u \tilde H(x_i^h,p_i^h,u_i^h)-\nabla_u \tilde H(x_h,p_h,u_h), 
	\hspace*{0.1cm}t\in[t_i,t_{i+1}), \hspace*{0.08cm}i=0,\dots,N-1.
	\end{align*}
	
	Repeating the calculations in the proof of Theorem \cite[Theorem 5.1]{Osmol+Vel-19}, we obtain 
	\be
	\max\left\lbrace \|\Delta_1\|_1,\|\Delta_2\|_1,\|\Delta_3\|_\infty\right\rbrace
	\le \max\left\lbrace 1,T\right\rbrace \tilde L(1+2\tilde M)h,
	\ee
	where $\tilde L,\tilde M$ are the numbers in Theorem \ref{Tstaa2}. We can choose $h_0>0$ depending 
	on $\tilde L, \tilde M, T$ and $b$ so that $\|\Delta_1\|_1+\|\Delta_2\|_1+\|\Delta_3\|_\infty\le b$ for all $h\le h_0$. 
	The claim follows from the SMsR property of 
	$\Phi_\pi$ with $C:=3\kappa(1+2\tilde M)\tilde L\max\left\lbrace 1,T\right\rbrace$. 
	The proof is complete since this holds for any arbitrary $\pi\in\mathcal H_\zeta$ satisfying Assumption (C1).
\end{proof}

\end{document}